\documentclass[10pt]{article}

\DeclareMathAlphabet{\mathpzc}{OT1}{pzc}{m}{it}

\usepackage{amsmath}
\usepackage{amssymb}
\usepackage{amsfonts}
\usepackage{amsthm}
\usepackage{mathrsfs}
\usepackage{ifsym}
\usepackage{fontenc}
\usepackage{textcomp}
\usepackage{tipa}
\usepackage{pxfonts}
\usepackage{enumerate}
\usepackage[pdftex]{color, graphicx}

\newcommand{\pa}{\parallel}

\allowdisplaybreaks

\newtheorem{thm}{Theorem}[section]

\newtheorem{defn}[thm]{Definition}

\newtheorem{cor}[thm]{Corollary}

\begin{document}

\title{A quadrilateral half-turn theorem}

\author{\renewcommand{\thefootnote}{\arabic{footnote}}
Igor Minevich and Patrick Morton}
\date{}
\maketitle

\abstract{If $ABC$ is a given triangle in the plane, $P$ is any point not on the extended sides of $ABC$ or its anticomplementary triangle, $Q$ is the complement of the isotomic conjugate of $P$ with respect to $ABC$, $DEF$ is the cevian triangle of $P$, and $D_0$ and $A_0$ are the midpoints of segments $BC$ and $EF$, respectively, a synthetic proof is given for the fact that the complete quadrilateral defined by the lines $AP, AQ, D_0Q, D_0A_0$ is perspective by a Euclidean half-turn to the similarly defined complete quadrilateral for the isotomic conjugate $P'$ of $P$ . This fact is used to define and prove the existence of a generalized circumcenter and generalized orthocenter for any such point $P$.}

\section{Introduction.}

The purpose of this note is to give a synthetic proof of the following surprising theorem.  We let $ABC$ be an ordinary triangle in the extended Euclidean plane, and we let $P$ be a point which does not lie on the sides of either $ABC$ or its anticomplementary triangle.  Furthermore, if $K$ denotes the complement map and $P'$ denotes the isotomic conjugate of $P$ with respect to $ABC$, then $Q=K(P')$ denotes the {\it isotomcomplement} of the point $P$ (Grinberg's terminology \cite{gr1}).  Furthermore, let $D_0, E_0, F_0$ be the midpoints of the sides of $ABC$ opposite $A, B$, and $C$, respectively.  \medskip

We denote by $T_P$ the unique affine map taking $ABC$ to the cevian triangle $DEF$ of $P$, and we set $A_0B_0C_0=T_P(D_0E_0F_0)$, the image of the medial triangle of $ABC$ under the map $T_P$.  Then $A_0, B_0, C_0$ are just the midpoints of segments $EF, DF$, and $DE$, respectively.  Also, $D_3E_3F_3$ is the cevian triangle of $P'$, so that $D_3$ is the reflection of the point $D$ across the midpoint $D_0$ of $BC$, etc.; $T_{P'}$ is the affine mapping for which $T_{P'}(ABC)=D_3E_3F_3$; and $A_0'B_0'C_0'=T_{P'}(D_0E_0F_0)$. (We are choosing notation to be consistent with the notation in \cite{mimo}, where the cevian triangles of $P$ and $Q$ are $DEF=D_1E_1F_1$ and $D_2E_2F_2$.) The theorem we wish to prove can be stated as follows.  \medskip

\begin{thm}[{Quadrilateral Half-turn Theorem}] If $Q'=K(P)$ is the isotomcomplement of $P'$, the complete quadrilaterals
$$\Lambda=(AP)(AQ)(D_0Q)(D_0A_0) \ \ and \ \ \Lambda'=(D_0Q')(D_0A_0')(AP')(AQ')$$
are perspective by a Euclidean half-turn about the point $N_1=$ midpoint of $AD_0 =$ midpoint of $E_0F_0$.  In particular, corresponding sides in these quadrilaterals are parallel.
\end{thm}

This shows that the symmetry between $P$ and $P'$, initially determined by different reflections across the midpoints of the sides of $ABC$, is also determined by a Euclidean isometry of the whole plane.  However, this isometry permutes the sides of $\Lambda$ and $\Lambda'$, so that side $AP$ in $\Lambda$ does not correspond to $AP'$ in $\Lambda'$, but to $D_0Q'$, and so forth.  There are similar statements corresponding to Theorem 1.1 for the quadrilaterals $(BP)(BQ)(E_0Q)(E_0B_0)$ and $(CP)(CQ)(F_0Q)(F_0C_0)$.

\begin{figure}[htbp]
\[\includegraphics[width=4in]{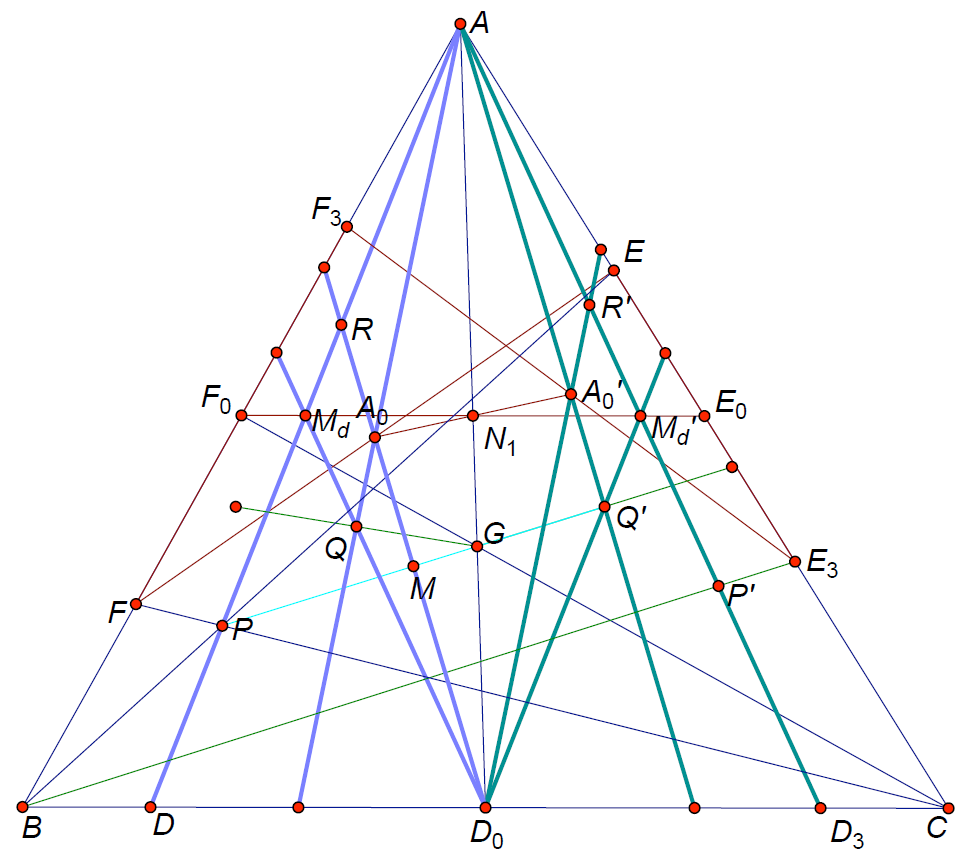}\]
\caption{Quadrilateral Half-turn Theorem}
\label{fig:1.3}
\end{figure}

\section{Preliminaries and proof.}

We require two results, for which synthetic proofs can be found in \cite{mimo}.

\begin{thm}[{Theorem 3 in \cite{gr1}}] Let $ABC$ be a triangle and $D, E, F$ the traces of point $P$ on the sides opposite $A, B$, and $C$. Let $D_0, E_0, F_0$ be the midpoints of the sides opposite $A, B, C$, and let $M_d, M_e, M_f$ be the midpoints of $AD, BE, CF$. Then $D_0M_d, E_0M_e, F_0M_f$ meet at the isotomcomplement $Q = K \circ \iota(P)$ of $P$.  ($\iota$ is the isotomic map.)
\label{thm:1.1}
\end{thm}

\begin{cor}
$D_0M_d = D_0Q$ is parallel to $AP'$ and $K(D_3) = M_d$.
\label{cor:1.1}
\end{cor}

See also Altshiller-Court \cite{ac} (p.165, Supp. Ex. 10).

\begin{thm}[{Grinberg-Yiu \cite{gr1}, \cite{y2}}] With $D, E, F$ as before, let $A_0$, $B_0$, $C_0$ be the midpoints of $EF, DF$, and $DE$, respectively. Then the lines $AA_0$, $BB_0$, $CC_0$ meet at the isotomcomplement $Q$ of $P$.
\label{thm:1.3}
\end{thm}

\begin{proof}[Proof of Theorem 1.1.]
(See Figure 1.) Let $R$ and $R'$ denote the midpoints of segments $AP$ and $AP'$, and $M_d$ and $M_d'$ the midpoints of segments $AD$ and $AD_3$, where $D_3=AP' \cdot BC$.  We first check that the vertices of the complete quadrilateral (see \cite{cox})
$$\Lambda = (AP)(AQ)(D_0Q)(D_0A_0)$$
are $A, R, M_d, Q, A_0$, and $D_0$.  It is clear that $A, Q, D_0$ are vertices.  Further, $M_d=AP \cdot D_0Q$ by Theorem 2.1 and $A_0=AQ \cdot D_0A_0$ by Theorem 2.3. \smallskip

We now show that $D_0, A_0$, and $R$ are collinear, from which we obtain $R=AP \cdot D_0A_0$.  Since $A_0E_0A_0'F_0$ joins the midpoints of the sides of the quadrilateral $FEE_3F_3$, it is a parallelogram, so the intersection of its diagonals is the point $A_0A_0' \cdot E_0F_0=N_1$.  Hence, $N_1$ bisects $A_0A_0'$ (and with $E_0F_0$ also $AD_0$).  \smallskip

Assume that $P$ is an ordinary point.  Let $M$ be the midpoint of $PQ'$; then $K(A)=D_0, K(Q')=M$ (since $K(P)=Q'$), so $AQ'$ is parallel to $D_0M$.  Now $R$ and $M$ are midpoints of sides in triangle $AQ'P$, so $RM$ is a line through $M=K(Q')$ parallel to $AQ'$, hence we have the equality of the lines $RM=D_0M=D_0R$.  If $T=A_0'N_1.D_0R$, then triangles $AN_1A_0'$ and $D_0N_1T$ are congruent ($\angle D_0TN_1 \cong \angle AA_0'N_1$ and AAS), so $N_1$ bisects $A_0'T$ and $T=A_0$.  (Note that $N_1$, as the midpoint of $E_0F_0$, lies on $AD_0$, and $A_0$ and $A_0'$ are on opposite sides of this line; hence $N_1$ lies between $A_0$ and $A_0'$.) This shows that $D_0, R$, and $A_0$ are collinear.  By symmetry, $D_0, A_0'$, and $R'$ are collinear whenever $P'$ is ordinary.  \smallskip

If $P'=Q$ is infinite, then $P$ is ordinary (it lies on the Steiner circumellipse of $ABC$), and we may use the congruence $AN_1A_0 \cong D_0N_1A_0'$ to get that $D_0A_0' || AA_0 =AQ$, which shows that $D_0, A_0'$, and $Q$ are collinear.  Thus, the last vertex of the quadrilateral
$$\Lambda'=(D_0Q')(D_0A_0')(AP')(AQ')$$
is $R'=AP' \cdot D_0A_0' =Q$ in this case. By symmetry, we get the same conclusion for $\Lambda$ when $P$ is infinite (in which case $P'$ is ordinary).\smallskip

Now consider the hexagon $AM_d'RD_0M_dQ'$ (if $P$ is ordinary).  Alternating vertices of this hexagon are on the lines $l=AP$ and $m=D_0Q'$, by Corollary 2.2, so the theorem of Pappus \cite{cox} implies that intersections of opposite sides, namely,
$$AM_d' \cdot D_0M_d, \ \ AQ' \cdot RD_0, \ \ \textrm{and} \ M_dQ' \cdot M_d'R,$$
are collinear.  The point $AM_d' \cdot D_0M_d=AP' \cdot D_0Q$ is on the line at infinity because $K(AP')=D_0Q$. By the above argument, $AQ' \cdot RD_0$ is also on the line at infinity.  Hence, $M_dQ'$ is parallel to $M_d'R$.  Since $Q'M_d'$ is parallel to $AP=M_dR$ (Theorem 2.1 and its corollary), $M_dQ'M_d'R$ is a parallelogram and the intersection of the diagonals $Q'R \cdot M_dM_d'$ is the midpoint of $M_dM_d'=K(DD_3)$ (Corollary 2.2).  But this midpoint is $N_1=K(D_0)$, since $D_0$ is the midpoint of $DD_3$.  Hence, $N_1$ also bisects $Q'R$, and by symmetry, $QR'$, when $P'$ is ordinary. \smallskip

We have shown that $N_1$ bisects the segments between pairs of corresponding vertices in the sets
$$\{A, R, M_d, Q, A_0, D_0\} \ \ \textrm{and} \ \ \{D_0, Q', M_d', R', A_0', A\}.$$  If $P'=Q$ is infinite, we replace $R'$ by $Q$ in the second set of vertices, and we get the same conclusion since $Q$ is then fixed by the half-turn about $N_1$.  This proves the theorem.
\end{proof}

\begin{cor}
a) If $P$ and $P'$ are ordinary, the Euclidean quadrilaterals $RA_0QM_d$ and $Q'A_0'R'M_d'$ are congruent.\\
b) If $P$ is ordinary, the points $D_0, R, A_0$, and $M=K(Q')$ are collinear, where $R$ is the midpoint of segment $AP$.  The point $M=K(Q')$ is the midpoint of segment $D_0R$.\\
c) If $P'$ is infinite, then $Q, M_d, D_0, A_0'$, and $K(A_0)$ are collinear.
\end{cor}

\begin{proof} Part a) is clear from the proof of the theorem.  For part b), we just have to prove the second assertion.  The theorem implies that quadrilateral $AQ'D_0R$ is a parallelogram, since $AQ'$ is parallel to $D_0A_0=D_0R$, $AR=AP$ is parallel to $D_0Q'$, and $R=AP \cdot D_0A_0$.  Thus, segment $AQ'$ is congruent to segment $D_0R$, and $D_0M=K(AQ')$ is half the length of $AQ' \cong D_0R$.  $M$ is clearly on the same side of line $D_0Q'$ as $P$ and $R$, so $M$ is the midpoint of $D_0R$.  Part c) follows by applying the complement map to the collinear points $P'=Q, D_3, A$, and $A_0$, to get that $Q, M_d, D_0$, and $K(A_0)$ are collinear, and then appealing to the argument in the fourth paragraph of the above proof, which shows that $A_0'$ is on $QD_0$.
\end{proof}

\section{An affine formula for the generalized orthocenter.}

To give an application of Theorem 1.1, we start with the following definition.

\begin{defn}
The point $O$ for which $OD_0 \pa QD, OE_0 \pa QE$, and $OF_0 \pa QF$ is called the {\bf generalized circumcenter} of the point $P$ with respect to $ABC$.  The point $H$ for which $HA \pa QD, HB \pa QE$, and $HC \pa QF$ is called the {\bf generalized orthocenter} of $P$ with respect to $ABC$.
\end{defn}

We have the following affine relationships between $Q, O$, and $H$.  We let $A_3'B_3'C_3'=T_{P'}(DEF)$ be the image of the cevian triangle $DEF$ of $P$ under the map $T_{P'}$.

\begin{thm}
\label{thm:HO}
The generalized circumcenter $O$ and generalized orthocenter $H$ exist for any point $P$ not on the extended sides of either $ABC$ or its anticomplementary triangle $K^{-1}(ABC)$, and are given by
$$O=T_{P'}^{-1}K(Q), \ \ H = K^{-1}T_{P'}^{-1}K(Q),$$
where $T_{P'}$ is the affine map taking $ABC$ to the cevian triangle $D_3E_3F_3$ of the point $P'$.
\end{thm}

 \noindent {\bf Remark.} The formula for the point $H$ can also be written as $H=T_L^{-1}(Q)$, where $L=K^{-1}(P')$ and $T_L$ is the map $T_P$ defined for $P=L$ and the anticomplementary triangle of $ABC$. \smallskip
 
\begin{proof}
We will show that the point $\tilde O=T_{P'}^{-1}K(Q)$ satisfies the definition of $O$, namely, that
$$\tilde OD_0 \pa QD, \ \ \tilde OE_0 \pa QE, \ \ \tilde OF_0 \pa QF.$$
It suffices to prove the first relation $\tilde OD_0 \pa QD$.  We have that
$$T_{P'}(\tilde OD_0)=K(Q)T_{P'}(D_0)=K(Q)A_0'$$
and
$$T_{P'}(QD)=P'A_3',$$
by \cite{mimo}, Theorem 3.7, according to which $T_{P'}(Q)=P'$, and by the definition of the point $A_3'=T_{P'}(D)$.  Thus, we just need to prove that $K(Q)A_0' \pa P'A_3'$.  We use the map $\mathcal{S}'=T_{P'}T_P$ from \cite{mimo}, Theorem 3.8, which takes $ABC$ to $A_3'B_3'C_3'$.  We have $\mathcal{S}'(Q)=T_{P'}T_P(Q)=T_{P'}(Q)=P'$, since $Q$ is a fixed point of $T_P$ (\cite{mimo}, Theorem 3.2).  Since $\mathcal{S}'$ is a homothety or translation, this gives that $AQ \pa \mathcal{S}'(AQ)=A_3'P'$.  Assuming that $P'$ is ordinary, we have $M'=K(Q)$, as in Corollary 2.4b), so by that result
$$K(Q)A_0'=M'A_0'=D_0A_0'.$$
Now Theorem 1.1 implies that $AQ \pa D_0A_0'$, and therefore $P'A_3' \pa K(Q)A_0'$.  This proves the formula for $O$.  To get the formula for $H$, just note that $K^{-1}(OD_0)=K^{-1}(O)A$, $K^{-1}(OE_0)=K^{-1}(O)B$, $K^{-1}(OF_0)=K^{-1}(O)C$ are parallel, respectively, to $QD,QE,QF$, since $K$ is a dilatation.  This shows that $K^{-1}(O)$ satisfies the definition of the point $H$. \medskip

If the point $P'=Q$ is infinite, then it is easy to see from Def. 3.1 that $O=H=Q$, and this agrees with the formulas of the theorem, since 
$$T_{P'}^{-1}K(Q)=T_{P'}^{-1}(Q)=K \circ [T_{P'}K]^{-1}(Q)=K \circ [T_{P'}K]^{-1}(P')=K(P')=Q,$$
using the fact that $T_{P'}K(P')=P'$ from \cite{mimo}, Theorem 3.7.
\end{proof}

\begin{cor}
If $P=Ge$ is the Gergonne point of $ABC$, $P'=Na$ is the Nagel point for $ABC$, and $Q=I$ is the incenter of $ABC$, the circumcenter and orthocenter of $ABC$ are given by the affine formulas
$$O=T_{P'}^{-1}K(Q), \ \ H = K^{-1}T_{P'}^{-1}K(Q).$$
\end{cor}

\noindent {\bf Remark.} In the corollary, $K(Q)$ is the Spieker center $X(10)$ of $ABC$, so the first formula says that $T_{Na}(O)=X(10)$.  See \cite{ki}.\medskip

We also prove the following relationship between the traces $H_a, H_b, H_c$ of $H$ on the sides $a=BC, b=CA, c=AB$ and the traces $D_2, E_2, F_2$ of $Q$ on those sides.

\begin{thm}
If the cevian triangles of $P$ and its isotomic conjugate $P'=\iota(P)$ for $ABC$ are $DEF$ and $D_3E_3F_3$, respectively, then we have the harmonic relations $\sf{H} (\it{DD_3,D_2H_a}), \sf{H} (\it{EE_3,E_2H_b}), \sf{H} (\it{FF_3,F_2H_c})$.  In other words, the point $H_a$ is the harmonic conjugate of $D_2$ with respect to the points $D, D_3$ on $BC$, with similar statements holding for the traces of $H$ and $Q$ on the other sides.
\end{thm}

\begin{proof}
Define the points $M=AH_a \cdot QD_3$ and $T=DQ \cdot AD_3$.  By Theorem 1.1, $QD_0 || AP'=AD_3$, so since $D_0$ is the midpoint of $DD_3$, it follows by considering triangle $DTD_3$ that $Q$ is the midpoint of $DT$.  By definition of $H$ we also have $DQ || AH_a$, so using similar triangles $DTD_3$ and $H_aAD_3$, we see that $M$ is the midpoint of $AH_a$.  Now project the points $H_aDD_2D_3$ on $BC$ from $Q$ to the points $H_aJ_\infty AM$ on $AH_a$, where $J_\infty=AH_a \cdot QD$ is on the line at infinity.  Then the relation $\sf{H}(\it{J_\infty M,AH_a})$ yields $\sf{H}(\it{DD_3,D_2H_a})$.
\end{proof}

In \cite{mm3} we will explore the properties of the points $O$ and $H$ in greater depth.  In order to give an example of the points $O$ and $H$, we give their barycentric coordinates in terms of the barycentric coordinates of the point $P=(x,y,z)$.  We note that
$$Q=(x',y',z')=(x(y+z),y(x+z),z(x+y)),$$
(see \cite{gr1}, \cite{mor}) while
$$K=\left(\begin{array}{ccc} 0 & 1 &1 \\1 & 0 & 1 \\1 & 1 & 0\end{array}\right), \ \ T_{P'}^{-1}=\left(\begin{array}{ccc} -xx' & yx' & zx' \\ xy' & -yy' & zy' \\ xz' & yz' & -zz' \end{array}\right).$$
From this and Theorem 3.2 we find that the barycentric coordinates of $O$ and $H$ are
\begin{eqnarray*}
O&=&(x(y+z)^2x'', y(x+z)^2y'', z(x+y)^2z''), \\
H &=&(xy''z'',yx''z'',zx''y'')=\left(\frac{x}{x''},\frac{y}{y''}, \frac{z}{z''}\right),
\end{eqnarray*}
where
$$x''=xy+xz+yz-x^2, \ \  y''=xy+xz+yz-y^2, \ \ z''=xy+xz+yz-z^2.$$
For example, using the coordinates of $O$ and $H$ it can be shown that if $P=Na$ is the Nagel point, then
\begin{eqnarray*}
O&=&(g(a,b,c), g(b,c,a), g(c,a,b))=X(6600),\\
&&\textrm{with} \ \ g(a,b,c)=a^2(b+c-a)(a^2+b^2+c^2-2ab-2ac),\\
H &=&(h(a,b,c),h(b,c,a),h(c,a,b))=X(6601),\\
&&\textrm{with} \ \ h(a,b,c)=(b+c-a)/(a^2+b^2+c^2-2ab-2ac).
\end{eqnarray*}
See \cite{ki}, \cite{mm}.  Here, $Na=(b+c-a, c+a-b,a+b-c)$, where $a,b,c$ are the side lengths of $ABC$.  (See \cite{mm}, where these points were given before being listed in \cite{ki}.)  Note that $H=\gamma(X(1617))$, where $\gamma$ is isogonal conjugation, and $X(1617)$ is the TCC-perspector of $X(57)=\gamma(X(9))=\gamma(Q)$.  We will generalize this fact in \cite{mm4}, by showing (synthetically) in general that $\gamma(H)$ is the TCC-perspector of $\gamma(Q)$.

\medskip

\noindent Dept. of Mathematics, Maloney Hall\\
Boston College\\
140 Commonwealth Ave., Chestnut Hill, Massachusetts, 02467-3806\\
{\it e-mail}: igor.minevich@bc.edu
\bigskip

\noindent Dept. of Mathematical Sciences\\
Indiana University - Purdue University at Indianapolis (IUPUI)\\
402 N. Blackford St., Indianapolis, Indiana, 46202\\
{\it e-mail}: pmorton@math.iupui.edu

\end{document}